\documentclass[12pt]{amsart}
\usepackage{amssymb}
\usepackage{color}
\usepackage{tikz-cd}
\usetikzlibrary{decorations.markings}

\usepackage{xcolor}
\definecolor{deepgreen}{cmyk}{0.99998,0,1,0}

\usepackage{hyperref}
\usepackage{comment}
\usepackage{mathrsfs}

\usepackage{geometry}
\usepackage[msc-links, lite]{amsrefs}

  \usepackage{bm}
  \usepackage{bbm}
  
  \usepackage[new]{old-arrows}
  
\theoremstyle{definition}

\newtheorem{theorem}{Theorem}[section]

\newtheorem{proposition}[theorem]{Proposition}
\newtheorem{lemma}[theorem]{Lemma}

\newtheorem{example}[theorem]{Example}
\newtheorem{problem}[theorem]{Problem}
\newtheorem{open question}[theorem]{Open Question}

\theoremstyle{definition}
\newtheorem{definition}[theorem]{Definition}

\numberwithin{equation}{section}
\theoremstyle{definition}

\newcommand{\pa}{\partial}

\newcommand{\bv}{\big\vert}

\newcommand{\cref}{\S\ \ref}
\newcommand{\Cref}{\S\ \ref}

\newcommand{\mb}{\mathbb}

\newcommand{\Sc}{\mathrm{Sc}}

\DeclareMathOperator{\Ric}{Ric}

\numberwithin{equation}{section}

\newcommand{\interior}[1]{%
	{\kern0pt#1}^{\mathrm{\,o}}%
}
\definecolor{pink}{RGB}{249,164,186}
\definecolor{grassgreen}{RGB}{128,255,0}

\numberwithin{equation}{section}
\title[Bounding the A-hat genus using scalar curvature]{Bounding the A-hat genus using scalar curvature lower bounds and isoperimetric constants}

\author{Qiaochu Ma$^1$}
\author{Jinmin Wang$^2$}
\author{Guoliang Yu$^1$}
\author{Bo Zhu$^1$}

\address{$^1$Department of Mathematics, Texas A\&M University}
\address{$^2$Institute of Mathematics, Chinese Academy of Sciences}

\date{}

\makeatletter
\newcommand*{\rom}[1]{\expandafter\@slowromancap\romannumeral #1@}
\makeatother

\begin{document} 
\clearpage

\begin{abstract}
In this paper, we prove an upper bound on the $\widehat{A}$  genus of a smooth, closed, spin Riemannian manifold using its scalar curvature lower bound, Neumann isoperimetric constant, and volume. The proof of this result relies on spectral analysis of the Dirac operator. We also construct an example to show that the Neumann isoperimetric constant in this bound is necessary. Our result partially answers a question of Gromov on bounding characteristic
numbers using scalar curvature lower bound.
\end{abstract}
\maketitle
\section{Introduction}

\subsection{Backgrounds}\label{A1''}
The scalar curvature $\Sc_g$ of an $m$-dimensional Riemannian manifold $(M, g)$ can be simply defined by the following formula:
\begin{equation}
    \frac{\mathrm{Vol}(B_M (x, r))}{\mathrm{Vol}(B_{ \mathbb{R}^m} (0, r))}= 1- \frac{\Sc_g (x)}{ 6(m+2)} r^2 +O(r^2),
\end{equation}
where the notation $\mathrm{Vol}$ stands for volume,   $x$ is a point in $M$, $B_M (x, r)$ is the geodesic ball in $M$ with radius $r$
and center $x$, and $B_{ \mathbb{R}^m} (0, r)$ is the ball with radius $r$ in the Euclidean space $\mathbb{R}^m$ with the standard Euclidean metric.

Scalar curvature encodes the information about the volume of small geodesic balls within a Riemannian manifold and plays a crucial role in both differential geometry and general relativity, such as the proof of the positive mass theorem by Schoen-Yau \cite{Schoen_Yau_pmt1}. This raises a question of what topological constraints scalar curvature imposes on a manifold. According to the classic Lichnerowicz theorem \cite{Lichnerowicz_formula}, a smooth, closed, spin Riemannian manifold with positive scalar curvature has a vanishing $\widehat{A}$ genus. 
When a spin manifold is simply connected, Stolz shows that the 
$\widehat{A}$ genus is the only obstruction to the existence of Riemannian metrics with positive 
scalar curvature in dimension greater than or equal to $5$ \cite{Stolz_simply_connected}. In the presence of a fundamental group, vanishing $\widehat{A}$ genus is not sufficient for the  existence of Riemannian metrics with positive scalar curvature. Schoen-Yau and Gromov-Lawson proved that the $m$-dimensional torus $\mathbb{T}^m$  does not admit any Riemannian metric with positive scalar curvature despite having  vanishing $\widehat{A}$ genus (see \cites{GromovLawson, Schoen_Yau_imcompressible, Schoen_Yau_structure_psc,Schoen_Yau_psc_higher} for  details). More generally, Rosenberg proves that the higher $\widehat{A}$ genera are obstructions to the existence of Riemannian metrics with positive 
scalar curvature if the strong Novikov conjecture holds \cites{Rosenberg_1,Rosenberg_psc_report}. Witten shows that for a $4$-dimensional manifold $M$, if $M$  has a non-zero Seiberg-Witten invariant, then $M$  admits no
 complete Riemannian metric of positive scalar curvature \cite{Witten_monopoles_four}. As a consequence, using the work of Taubes \cite{Taubes_Seiberg_Witten}, there exists a simply connected $4$-dimensional manifold with vanishing $\widehat{A}$ genus that admits no  Riemannian metric with positive scalar curvature. For foliated manifolds, Connes, Tang-Willett-Yao, and Zhang used  higher index theoretic invariants to prove striking nonexistence results  for Riemannian metrics with leafwise positive scalar curvature \cite{Zhang_psc_foliation,Connes_transverse_fundamental_class, TWY_Roe}.

 Kazdan-Warner proved a surprising result stating on a compact manifold of dimension greater than or equal to $3$, every smooth function which is
negative somewhere, is the scalar curvature of some complete Riemannian metric  \cite{Kazdan_Warner_scalar_curvature}. 
Gromov demonstrated that having a scalar curvature bounded below by a (possibly negative) constant $k_0$ 
  is a $C^0$
  property (see \cite[\S\,1.8, $C^0$-Limit Theorem]{Gromov_C0}, Bamler provided an alternative proof of this result using Ricci flow \cite{Bamler_C0_RF_proof}). This means that if a sequence of smooth complete Riemannian metrics with  scalar curvature lower bound $k_0$
  converges to a smooth Riemannian metric in the $C^0$
  sense, the scalar curvature of the limit metric also maintains the same lower bound $k_0$. In this article we investigate the topological complexity of spin Riemannian manifolds with scalar curvature lower bound. Specifically, we show that the $\widehat{A}$ genus of any closed spin Riemannian manifold can be constrained by the  scalar curvature lower bound, the Neumann isoperimetric constant and the manifold's volume. We also illustrate that this result is optimal by constructing an example to 
  demonstrate that the Neumann isoperimetric constant is necessary.  Our result partially answers a question of Gromov on  bounding
  characteristic numbers using scalar curvature lower bound (see \cite[\S\,3.13, Charateristic Numbers Conjecture]{Gromov4lectures}).

In the three dimensional case, 
Perelman’s monotonicity formula for the
Ricci flow with surgery implies
that simplicial volume, an invariant describing  topological complexity of the fundamental class in homology theory, can be bounded by the scalar curvature lower bound (see Agol-Storm-William \cite[Theorem 3.1]{Agol_simplicial_volume}). In another direction, Braun-Sauer  prove that the Euler characteristic, simplicial volume and  $\ell^2$-Betti numbers of a closed aspherical Riemannian manifold can be bounded by a lower bound on the macroscopic scalar curvature and its volume \cite[Theorem 1.4]{Sabine_Roman_volume}.  Recall that macroscopic scalar curvature is a large scale analogue of scalar curvature and measures the volumes of relatively large balls within the manifold, rather than the smaller balls used to define standard scalar curvature at a point. 

\subsection{Main result}
Let us fix an integer $m\geqslant 3$. Let $M$ be a spin manifold with $\dim_\mathbb{R}M=m$. We equip $M$ with an Riemannian metric $g$, then let $dv_M$ be the volume form on $M$, $\mathrm{Vol}(M)$ the volume of $M$, $\nabla^{TM}$ the Levi-Civita connection on $TM$, $R^{TM}$ the curvature and $\mathrm{Sc}_g$ be the scalar curvature. Let $\widehat{A}(M)$ denote the A-hat genus of $M$ and $IN(M)$ be the Neumann isoperimetric constant of $M$ (see \eqref{2.3} and \eqref{2.7'} for definitions).

 Now we state our main result.

\begin{theorem}\label{T1.1} If a closed spin Riemannian  manifold $M$ with dimension $m$ has scalar curvature $\Sc_g\geqslant k_0$ for some constant $k_0$, then
there exists a constant $C_m>0$ depending only on $m$ such that
\begin{equation}\label{1.1}
\bv\widehat{A}(M)\bv\leqslant C_m\Big(\frac{\vert k_0\vert^{m/2}\mathrm{Vol}(M)}{IN(M)^m}+1\Big).
\end{equation}
\end{theorem}

Note that in the upper bound of above theorem, both $\vert k_0\vert^{m/2}\mathrm{Vol}(M)$ and $IN(M)$ are invariant under rescaling of the Riemannian metric.

We shall  provide examples showing that the Neumann isoperimetric constant $IN(M)$ in the above inequality \eqref{1.1} is essential. 

In \cite{Gallot_isoperimetric}*{Theorem 11}, Gallot proved that  $\widehat{A}(M)$ can be bounded
by a lower bound on Ricci curvature and the diameter  of $M$. By \cite{Gallot_isoperimetric}*{Theorem 3}, lower bounds on Ricci curvature and diameter of a Riemannian manifold  can be used to provide a positive lower bound 
for the Neumann isoperimetric constant (with the additional help of its volume).
This implies that Theorem \ref{1.1} can be viewed as a generalization of the finiteness  result on $\widehat{A}(M)$ given by \cite{Gallot_isoperimetric}*{Theorem 11}.

\subsection{Acknowledgment}

We are grateful to Professor Ruobing Zhang  for enlightening discussions regarding the Sobolev inequality and the Neumann isoperimetric constant and for bringing the relevant sections of Li \cite{MR2962229} to our attention. We also thank Professor Guofang Wei and Professor Zhizhang Xie for inspiring discussions and for pointing out that our result can be considered as a generalization of  Gallot \cite{Gallot_isoperimetric}*{Theorem 11}.  The third author is partially supported by NSF 2247322
and the fourth author is partially supported by the AMS Simons Travel Grant.

\section{Preliminaries}\label{A}

In this section, we review some basic concepts and tools that will be used in the proof of the main theorem of this article.

\subsection{Lichnerowicz formula and Atiyah-Singer index theorem}

Let $S_M=S_M^+\oplus S_M^-$ be the spinor bundle of $M$ with the canonical Hermitian metric $h^{S_M}$ and the Clifford connection $\nabla^{S_M}$. Let $D^{S_M}$ be the Dirac operator and $\Delta^{S_M}$ the nonnegative Laplacian, both acting on $C^\infty(M,S_M)$. 

We recall the classical Lichnerowicz formula in \cite{Lichnerowicz_formula} as follows.
\begin{lemma}We have
\begin{equation}\label{2.1}
	\big(D^{S_M}\big)^2=\Delta^{S_M}+\frac{\mathrm{Sc}_g}{4}.
\end{equation}
\end{lemma}
For $D^{S_M^\pm}=D^{S_M}|_{C^\infty(M,S_M^\pm)}$, we define the Fredholm index of $D^{S_M^+}$ by
\begin{equation}\label{2.2}
		\mathrm{Ind}\big(D^{S_M^+}\big)=\dim_\mathbb{C}\ker \big(D^{S_M^+}\big)-\dim_\mathbb{C}\ker \big(D^{S_M^-}\big).
\end{equation}

A remarkable property is that the Fredholm index is invariant under continuous deformation of the operator. It is this property that makes the Fredholm index computable.

The A-hat genus $\widehat{A}(M)$ of $M$ is defined by
\begin{equation}\label{2.3}
\widehat{A}(M)=(4\pi\sqrt{-1})^{-m/2}\int_M\mathrm{Tr}^{TX}\Big[{\det}^{1/2}\Big(\frac{R^{TM}}{\sinh (R^{TM})}\Big)\Big].
\end{equation}
The Atiyah-Singer index formula \cite{MR236950,MR236952} reads as follows, see also Berline-Getzler-Vergne \cite[Theorem 4.9]{MR2273508}.
\begin{theorem}We have
\begin{equation}\label{2.4}
		\mathrm{Ind}\big(D^{S_M^+}\big)=\widehat{A}(M).
\end{equation}
\end{theorem}

\subsection{Kato's inequality}

Let $\Delta^M$ be the nonnegative Laplacian acting on $C^\infty(M)$. Now we recall the Kato's inequality for vector bundles by Hess-Schrader-Uhlenbrock \cite[proposition 2.2]{MR602436}, which we apply here specifically to $(M, S_M)$.
\begin{proposition}
For any $s\in C^\infty(M,S_M)$, we have the following inequality in the space $\mathscr{D}'(M)$ of distributions
\begin{equation}\label{2.5}
	\Delta^M\Vert s\Vert_{S_M}\leqslant\mathrm{Re}\langle \Delta^{S_M}s,\mathrm{sign}_\xi s\rangle_{S_M},
\end{equation}
where $\xi$ is an arbitrary measurable section of the unit sphere bundle of $S_M$ and
\begin{equation}
\mathrm{sign}_\xi s=\begin{cases}s/\Vert s\Vert_{S_M}\ \ &\text{on}\ \mathrm{supp}(s),\\
	\xi\ \ &\text{otherwise}.
\end{cases}
\end{equation}
\end{proposition}

Let $H^1_M$ be the Sobolev space on $M$ and $H^{-1}_M$ its dual through the $L^2$-inner product $\langle\cdot,\cdot\rangle_{L^2(M,S_M)}$. We note that since $\Vert s\Vert_{S_M}\in H^1_M$, \eqref{2.5} indeed holds in $H^{-1}_M$.

\subsection{The Neumann isoperimetric constant}

The Neumann isoperimetric constant of $M$ is defined by
\begin{equation}\label{2.7'}
IN(M)=\inf_{\begin{subarray}{c}
	\pa\Omega_1=S=\pa\Omega_2,\\ M=\Omega_1\cup S\cup\Omega_2
\end{subarray}}\frac{\mathrm{Area}(S)}{\min\{\mathrm{Vol}({\Omega_1}),\mathrm{Vol}({\Omega_2})\}^{(m-1)/m}},
\end{equation}
where the infimum is taken over all hypersurfaces $S$
dividing $M$ into two parts: $\Omega_1$ and $\Omega_2$, and $\mathrm{Area}(S)$ is the $(m-1)$-dimensional volume of $S$.

The Neumann isoperimetric constant here corresponding to the one obtained by taking $\alpha=m/(m-1)$ in \cite[Definition 9.2]{MR2962229} or see Dai-Wei-Zhang \cite{Dai_Wei_Zhang_Neumann}*{Definition 1.1}. This Neumann isoperimetric constant here remains unchanged when the Riemannian metric is rescaled, thanks to the specific choice of the power in the denominator of its definition. Additionally, we note that the Neumann isoperimetric constant is continuous with respect to $C^0$  convergence of Riemannian metrics.
Finally we mention that intuitively, the Neumann isoperimetric constant of a Riemannian manifold $M$ is small if there is a thin neck.

We have the following Poincaré-Wirtinger inequality, see \cite[Corollary 9.9]{MR2962229}.
\begin{proposition}
There are $C_{1,m}, C_{2,m}>0$ depending only on $m$ such that for any $f\in H^1_M$, we have
\begin{equation}\label{2.8}
	\begin{split}
	\int_M\vert\nabla f\vert^2dv_M\geqslant C_{1,m}IN(M)^2\bigg(\Big(&\int_M\vert f\vert^{2m/(m-2)}dv_M\Big)^{(m-2)/m}\\
	&-C_{2,m}\mathrm{Vol}(M)^{-2/m}\int_M\vert f\vert^2dv_M\bigg).
	\end{split}
\end{equation}
\end{proposition}

\section{Proof of the main theorem}

In this section, we give the proof for the main theorem of this article.  As a preparation, we apply Moser's iteration technique \cite{Moser_Harnack} to prove the following proposition. Here we closely follow the proof of  \cite[Theorem 10.2]{MR2962229}.
\begin{proposition} Assume $\Sc_g\geqslant k_0$.
There is $C_{3,m}>0$ depending only on $m$ such that for any $s\in \ker (D^{S_M})$, we have
\begin{equation}\label{3.1}
	\begin{split}
		\Vert s\Vert_{L^{\infty}(M,S_M)}\leqslant C_{3,m}\Big(\frac{\vert k_0\vert^{m/4}}{IN(M)^{m/2}}+\mathrm{Vol}(M)^{-1/2}\Big)\Vert s\Vert_{L^{2}(M,S_M)}.
	\end{split}
\end{equation}
\end{proposition}
\begin{proof}

For $s\in \ker (D^{S_M})$, by \eqref{2.1} we have
\begin{equation}
\Delta^{S_M}s=-\frac{\Sc_g}{4} s,
\end{equation}
which, together with \eqref{2.5}, implies the following inequality in $H^{-1}_M$
\begin{equation}\label{3.3}
\Delta^M \big(\Vert s\Vert_{S_M}\big)\leqslant \frac{\vert k_0\vert}{4}\Vert s\Vert_{S_M}.
\end{equation}

For any $a\geqslant 2$ and nonnegative $u\in H^1_M$, integrating by parts implies
\begin{equation}\label{3.4}
	\begin{split}
		\int_Mu^{a-1}\Delta^M udv_M&=(a-1)\int_Mu^{a-2}\big\vert\nabla u\big\vert^2dv_M\\
		&=\frac{4(a-1)}{a^2}\int_M\big\vert\nabla\big(u^{a/2}\big)\big\vert^2dv_M.
 	\end{split}
\end{equation}

Taking $u=\Vert s\Vert_{S_M}$ in \eqref{3.4} and then $f=\Vert s\Vert_{S_M}^{a/2}$ in \eqref{2.8}, by \eqref{3.3}, we deduce that
\begin{equation}
\begin{split}
\frac{\vert k_0\vert}{4}\int_M\Vert s\Vert_{S_M}^adv_M\geqslant \frac{C_{1,m}IN(M)^2}{a}\bigg(\Big(&\int_M\Vert s\Vert_{S_M}^{am/(m-2)}dv_M\Big)^{(m-2)/m}\\
&-C_{2,m}\mathrm{Vol}(M)^{-2/m}\int_M\Vert s\Vert_{S_M}^{a}dv_M\bigg),
\end{split}
\end{equation}
or equivalently,
\begin{equation}\label{3.6}
\begin{split}
    \Big(\frac{a\vert k_0\vert}{4C_{1,m}IN(M)^2}+C_{2,m}\mathrm{Vol}(M)^{-2/m}\Big)&\int_M\Vert s\Vert_{S_M}^adv_M\\
    &\geqslant\Big(\int_M\Vert s\Vert_{S_M}^{am/(m-2)}dv_M\Big)^{(m-2)/m}.
\end{split}
\end{equation}

By setting $a=2b^\ell$ for $b=m/(m-2),\ell\in\mathbb{N}$ in \eqref{3.6}, we obtain
\begin{equation}
\Vert s\Vert_{L^{2b^{\ell+1}}(M,S_M)}\leqslant\Big(\frac{a\vert k_0\vert}{4C_{1,m}IN(M)^2}+C_{2,m}\mathrm{Vol}(M)^{-2/m}\Big)^{1/2b^\ell}\Vert s\Vert_{L^{2b^{\ell}}(M,S_M)},
\end{equation}
which indicates that
\begin{equation}
	\begin{split}
\Vert s\Vert_{L^{\infty}(M,S_M)}&=\lim_{\ell\to\infty}\Vert s\Vert_{L^{2b^{\ell}}(M,S_M)}\\
&\leqslant b^{\sum_{\ell=0}^{\infty}\ell/2b^\ell}\Big(\frac{\vert k_0\vert}{2C_{1,m}IN(M)^2}+C_{2,m}\mathrm{Vol}(M)^{-2/m}\Big)^{m/4}\Vert s\Vert_{L^{2}(M,S_M)}.
	\end{split}
\end{equation}
To see this, it is sufficient to calculate that
\begin{equation}
	\begin{split}
\prod_{\ell=0}^{\infty}\Big(&\frac{b^\ell\vert k_0\vert}{2C_{1,m}IN(M)^2}+C_{2,m}\mathrm{Vol}(M)^{-2/m}\Big)^{1/2b^\ell}\\
&\leqslant\prod_{\ell=0}^{\infty}\Big(\frac{b^{\ell}\vert k_0\vert}{2C_{1,m}IN(M)^2}+b^{\ell}C_{2,m}\mathrm{Vol}(M)^{-2/m}\Big)^{1/2b^\ell}\\
&=b^{\sum_{\ell=0}^{\infty}\ell/2b^\ell}\Big(\frac{\vert k_0\vert}{2C_{1,m}IN(M)^2}+C_{2,m}\mathrm{Vol}(M)^{-2/m}\Big)^{\sum_{\ell=0}^{\infty}1/2b^\ell}\\
&=\Big(\frac{m}{m-2}\Big)^{(m^2-2m+4)/4}\Big(\frac{\vert k_0\vert}{2C_{1,m}IN(M)^2}+C_{2,m}\mathrm{Vol}(M)^{-2/m}\Big)^{m/4},
	\end{split}
\end{equation}
hence, \eqref{3.1} has been derived.
\end{proof}

By \eqref{2.2} and \eqref{2.4}, we see that \eqref{1.1} follows immediately from the following result. We shall imitate the proof of Li \cite[Lemma 7.2]{MR2962229}.

\begin{proposition} Assume that $\Sc_g\geqslant k_0$, then there is constant $C_{m}>0$ depending only on $m$ such that
\begin{equation}\label{3.10}
\dim_\mathbb{C}\ker \big(D^{S_M}\big)\leqslant C_m\Big(\frac{\vert k_0\vert^{m/2}\mathrm{Vol}(M)}{IN(M)^m}+1\Big).
\end{equation}
\end{proposition}

\begin{proof}
Let us take an orthonormal basis $\{s_i\}$ of $\ker(D^{S_M})$. For any $x\in M$, let $\{v_k\}$ be an orthonormal basis of the fiber $S_{M,x}$ at $x$, then we take $s_{x,v_k}\in C^\infty(M,S_M)$ by
\begin{equation}
s_{x,v_k}=\sum_i \langle v_k,s_i(x)\rangle_{S_M}s_i.
\end{equation}
For any $s\in L^2(M,S_M)$, we denote its projection to $\ker(D^{S_M})$ by $p_{\ker(D^{S_M})}s$, that is,
\begin{equation}
p_{\ker(D^{S_M})}s=\sum_{i}\langle s, s_i\rangle_{L^2(M,S_M)}s_i,
\end{equation}
then we compute that
\begin{equation}
	\begin{split}
	\langle s_{x,v_k},s\rangle_{L^2(M,S_M)}&=\sum_i \langle v_k,s_i(x)\rangle_{S_M}\langle s_i,s\rangle_{L^2(M,S_M)}\\
	&=\big\langle v_k,(p_{\ker(D^{S_M})}s)(x)\big\rangle_{S_M},
	\end{split}
\end{equation}
which, together with \eqref{3.1}, implies that
\begin{equation}
	\begin{split}
\big\vert\langle s_{x,v_k},s\rangle_{L^2(M,S_M)}\big\vert&\leqslant \big\Vert p_{\ker(D^{S_M})}s\big\Vert_{C^0(M,S_M)}\\
&\leqslant C_{3,m}\Big(\frac{\vert k_0\vert^{m/4}}{IN(M)^{m/2}}+\mathrm{Vol}(M)^{-1/2}\Big)\big\Vert p_{\ker(D^{S_M})}s\big\Vert_{L^2(M,S_M)}\\
&\leqslant C_{3,m}\Big(\frac{\vert k_0\vert^{m/4}}{IN(M)^{m/2}}+\mathrm{Vol}(M)^{-1/2}\Big)\Vert s\Vert_{L^2(M,S_M)}.
	\end{split}
\end{equation}
Therefore, we have
\begin{equation}
\Vert s_{x,v_k}\Vert_{L^2(M,S_M)}\leqslant C_{3,m}\Big(\frac{\vert k_0\vert^{m/4}}{IN(M)^{m/2}}+\mathrm{Vol}(M)^{-1/2}\Big).
\end{equation}
In other words,
\begin{equation}
	\sqrt{\sum_{i}\big\vert\langle v_k,s_i(x)\rangle_{S_M}\big\vert^2}\leqslant C_{3,m}\Big(\frac{\vert k_0\vert^{m/4}}{IN(M)^{m/2}}+\mathrm{Vol}(M)^{-1/2}\Big).
\end{equation}
Squaring and summing over $k$, we see that
\begin{equation}
\sum_{i}\Vert s_i(x)\Vert^2_{S_M}\leqslant \dim_\mathbb{C}S_M\cdot C_{3,m}^2\Big(\frac{\vert k_0\vert^{m/4}}{IN(M)^{m/2}}+\mathrm{Vol}(M)^{-1/2}\Big)^2,
\end{equation}
then integrating over $M$, we deduce \eqref{3.10}.  
\end{proof} 

\section{Examples of Riemannian manifolds with large A-hat genus, small scalar curvature, and bounded volume}

In this section, we construct a sequence of  closed Riemannian manifolds with bounded volume, arbitrarily large $\widehat{A}$ genus, and arbitrarily small scalar curvature. These examples demonstrate that the Neumann isoperimetric constant cannot be removed in  Theorem \ref{T1.1} for general Riemannian manifolds.

Let us recall the following proposition essentially due to Gromov-Lawson \cite[Proposition 5.1]{Gromov_Lawson_fundamental_group} and Schoen-Yau \cite[Corollary 7]{Schoen_Yau_structure_psc}. However, the following state is a quantitative version due to Sweeney \cite[Proposition 4.2]{sweeney2023examplesscalarspherestability}.

\begin{proposition} \label{proposition: connected_sum_scalar_volume}
    Given  two Riemannian manifolds $(M,g)$ and $(M', g')$, then for any $\varepsilon>0$, there exists a complete Riemannian metric $g^\#$ on the connected sum $M\# M'$ such that
    \begin{enumerate}
        \item $\Sc_{g^\#} \geqslant \min\big\{\min_{x \in M}\Sc_{g}(x), \min_{x' \in M'}\Sc_{g'}(x')\big\} -{\varepsilon}$,
        \item $|\mathrm{Vol}({M\# M'}) - \mathrm{Vol}({M}) - \mathrm{Vol}({M'})| < {\varepsilon}$.
    \end{enumerate}
\end{proposition}
\begin{example} 
The Kummer surface $K$, the hyperplane in the complex projective space $\mb {CP}^{3}$  given by
the equation $ z_1^4+z_2^4+z_3^4 +z_4^4=0$ , is spin and has 
$\widehat{A}(K) =2$, and hence does not
admit a metric of positive scalar curvature.
More generally if $K$ is a smooth, spin, complex hypersurface in $\mb{CP}^{n+1}$ of degree $d = n+2$,  
    then
    \begin{itemize}
        \item  by Yau's solution of Calabi conjecture in \cite{Yau_Solution_to_Calabi},  $K$ carries a metric $g$ with $\Sc_g =0$, in fact, with $\Ric_g = 0$,

        \item 
        we have $\widehat{A}(K)\neq 0$ for even $n$.
    \end{itemize}
 We now assume that the Riemannian manifold $(M_1,g_1)$  has scalar curvature $\Sc_{g_1} = 0$. By rescaling we can further assume that $\mathrm{Vol}(M_1)=\frac{1}{C}$ for a $C>0$ to be defined. We consider the connected sum $M_2 =M_1\# M_1$. Then by Proposition \ref{proposition: connected_sum_scalar_volume}, there exists a complete, smooth Riemannian metric $g_2$ on $M_2$ such that
 \begin{itemize}
     \item $\Sc_{g_2} \geqslant - \frac{\varepsilon}{2}$ on $M_2$,
     \item $\mathrm{Vol}({M_2}) \leqslant \frac{2}{C} + \frac{\varepsilon}{2}$.
 \end{itemize}
 Note that $\hat{A}(M_2) = 2\hat{A}(M_1)$. Moreover, we consider $M_4 = M_2 \# M_2$. Similarly, by Proposition \ref{proposition: connected_sum_scalar_volume}, we can construct a complete Riemannian manifold $(M_4, g_4)$ such that
 \begin{itemize}
     \item $\Sc_{g_4} \geqslant -( \frac{\varepsilon}{2} +\frac{\varepsilon}{2^2})$ on $M_4$,
     \item $\mathrm{Vol}({M_4}) \leqslant \frac{4}{C} + (\frac{\varepsilon}{2}  + \frac{\varepsilon}{2^2})$.
 \end{itemize}
 Hence, we can construct a complete Riemannian manifold $(M_{2^\ell}, g_{2^\ell})$ inductively by $$M_{2^\ell} = M_{2^{\ell-1}} \#  M_{2^{\ell-1}}$$ 
 such that 
  \begin{itemize}
     \item $\Sc_{g_{2^\ell}} \geqslant -( \frac{\varepsilon}{2} +\frac{\varepsilon}{2^2} + \cdots + \frac{\varepsilon}{2^\ell})$ on $M_{2^\ell}$,
     \item $\mathrm{Vol}({M_{2^\ell}}) \leqslant \frac{2^\ell}{C} + (\frac{\varepsilon}{2}  + \frac{\varepsilon}{2^2})$.
 \end{itemize}
 Note that $\widehat{A}(M_{2^\ell}) = 2^\ell\widehat{A}(M)$. Moreover, if we take $C= 2^{\ell+1}$ and set $\varepsilon < \frac{1}{2}$, then $\mathrm{Vol}({M_{2^\ell}}) \leqslant 1$ for all $\ell$. 

 When $\ell$ is large enough, the Riemannian manifolds  $(M_{2^\ell}, g_{2^\ell})$ have
 bounded volume, arbitrarily  large $\widehat{A}$ genus, and arbitrarily small scalar curvature.
 These examples show that the Neumann isoperimetric constant in Theorem \ref{T1.1} is necessary.
\end{example}

\section{An open problem}

Finally we raise an open problem on the Neumann isoperimetric constant.

\begin{definition}
Let $k_0$ be a constant. We define

$$IN_{k_0} (M) = \sup ~ \{ IN(M, g)\mid \Sc_g \geqslant k_0, \mathrm{Vol}_g (M)\leqslant 1\},$$
where $ IN(M, g) $ is the Neumman isoperimetric constant of a Riemannian manifold $(M,g)$, $\mathrm{Vol}_g (M)$ is the volume of $(M, g)$,  and the superemum is taken over all Riemannian metrics $g$ over $M$ whose scalar curvature is greater than or equal to $k_0$ and volume is at most $1$.

\end{definition}

While the invariant $IN_{k_0} (M)$ may be hard to compute, a solution to the following problem can be helpful in estimating it.

\begin{problem}
Use conformal changes of a given Riemannian metric to give a lower bound
for $IN_{k_0} (M)$. More precisely, given a Riemannian metric $g_0$  on $M$ and a constant $k_0$, find a lower bound for $$IN_{k_0, g_0} (M) = \sup ~ \{ IN(M, g)\mid g=e^fg_0, \Sc_g \geqslant k_0, \mathrm{Vol}_g (M)\leqslant 1\},$$
where the supremum is taken over all smooth functions $f$ on $M$ 
such that the Riemannian metric $g=e^fg_0$ has  scalar curvature greater than or equal to $k_0$ and volume at most $1$.
\end{problem}

The motivation of finding such a lower bound is that it can be used to improve the upper bound for  $\widehat{A}$ genus of $M$ in Theorem 1.1.
The case of aspherical manifolds is of particular interest.

We remark that  the scalar curvature can be computed explicitly after a conformal change of a Riemannian metric. If $g=e^fg_0$, we then have
$$\Sc_g =e^{-f} (    \Delta_{g_0} f +\Sc_{g_0}    ),$$
where  $ \Delta_{g_0}$ is the Laplacian associated to the Riemannian metric $g_0$.

\nocite{WXY_decay,Wang:2021um,Wang:2021tq}


\addcontentsline{toc}{section}{References}


\end{document}